\documentclass[12pt]{amsart}

\usepackage{graphicx}
\usepackage{amssymb}
\usepackage{amsthm}
\usepackage{amsmath}
\usepackage{stmaryrd}
\usepackage{cite}
\usepackage[mathscr]{eucal}
\usepackage{hyperref}
\usepackage[margin=1in]{geometry}
\usepackage{comment}

\usepackage{url}
\usepackage{hyperref}

\title[Expanding gradient Ricci solitons]{New expanding Ricci solitons \\ starting in dimension four}
\author[Nienhaus and Wink]{Jan Nienhaus and Matthias Wink}

\address{Mathematisches Institut, Universit\"at M\"unster, Einsteinstra{\ss}e 62, 48149 M\"unster}
\email{j.nienhaus@uni-muenster.de}
\email{mwink@uni-muenster.de}

\keywords{Ricci solitons, Einstein metrics, warped products}

\makeatletter
\@namedef{subjclassname@2020}{
\textup{2020} Mathematics Subject Classification}
\makeatother

\subjclass[2020]{53C25, 53E20}

\begin{document}
\newcommand{\Ext}{\bigwedge\nolimits}
\newcommand{\Div}{\operatorname{div}}
\newcommand{\Hol} {\operatorname{Hol}}
\newcommand{\diam} {\operatorname{diam}}
\newcommand{\Scal} {\operatorname{Scal}}
\newcommand{\scal} {\operatorname{scal}}
\newcommand{\Ric} {\operatorname{Ric}}
\newcommand{\Hess} {\operatorname{Hess}}
\newcommand{\grad} {\operatorname{grad}}
\newcommand{\Sect} {\operatorname{Sect}}
\newcommand{\Rm} {\operatorname{Rm}}
\newcommand{ \Rmzero } {\mathring{\Rm}}
\newcommand{\Rc} {\operatorname{Rc}}
\newcommand{\Curv} {S_{B}^{2}\left( \mathfrak{so}(n) \right) }
\newcommand{ \tr } {\operatorname{tr}}
\newcommand{ \id } {\operatorname{id}}
\newcommand{ \Riczero } {\mathring{\Ric}}
\newcommand{ \ad } {\operatorname{ad}}
\newcommand{ \Ad } {\operatorname{Ad}}
\newcommand{ \dist } {\operatorname{dist}}
\newcommand{ \rank } {\operatorname{rank}}
\newcommand{\Vol}{\operatorname{Vol}}
\newcommand{\dVol}{\operatorname{dVol}}
\newcommand{ \zitieren }[1]{ \hspace{-3mm} \cite{#1}}
\newcommand{ \pr }{\operatorname{pr}}
\newcommand{\diag}{\operatorname{diag}}
\newcommand{\Lagr}{\mathcal{L}}
\newcommand{\av}{\operatorname{av}}
\newcommand{ \floor }[1]{ \lfloor #1 \rfloor }
\newcommand{ \ceil }[1]{ \lceil #1 \rceil }
\newcommand{\Sym} {\operatorname{Sym}}
\newcommand{\bcirc}{ \ \bar{\circ} \ }
\newcommand{\conj}[1]{ \overline{ #1 } }
\newcommand{\sign}[1]{\operatorname{sign}(#1)}
\newcommand{\cone}{\operatorname{cone}}
\newcommand{\pbd}{\varphi_{bar}^{\delta}}

\newtheorem{theorem}{Theorem}[section]
\newtheorem{definition}[theorem]{Definition}
\newtheorem{example}[theorem]{Example}
\newtheorem{remark}[theorem]{Remark}
\newtheorem{lemma}[theorem]{Lemma}
\newtheorem{proposition}[theorem]{Proposition}
\newtheorem{corollary}[theorem]{Corollary}
\newtheorem{assumption}[theorem]{Assumption}
\newtheorem{acknowledgment}[theorem]{Acknowledgment}
\newtheorem{DefAndLemma}[theorem]{Definition and lemma}

\newcommand{\R}{\mathbb{R}}
\newcommand{\N}{\mathbb{N}}
\newcommand{\Z}{\mathbb{Z}}
\newcommand{\Q}{\mathbb{Q}}
\newcommand{\C}{\mathbb{C}}
\newcommand{\F}{\mathbb{F}}
\newcommand{\X}{\mathcal{X}}
\newcommand{\D}{\mathcal{D}}
\newcommand{\Cont}{\mathcal{C}}

\renewcommand{\labelenumi}{(\alph{enumi})}
\newtheorem{maintheorem}{Theorem}[]
\renewcommand*{\themaintheorem}{\Alph{maintheorem}}
\newtheorem*{theorem*}{Theorem}
\newtheorem*{corollary*}{Corollary}
\newtheorem*{remark*}{Remark}
\newtheorem*{example*}{Example}
\newtheorem*{question*}{Question}
\newtheorem*{definition*}{Definition}
\newtheorem{conjecture}[maintheorem]{Conjecture}

\begin{abstract}
We prove that there exists a gradient expanding Ricci soliton asymptotic to any given cone over the product of a round sphere and a Ricci flat manifold. In particular we obtain asymptotically conical expanding Ricci solitons with positive scalar curvature on $\R^3 \times S^1.$ More generally we construct continuous families of gradient expanding Ricci solitons on trivial vector bundles over products of Einstein manifolds with arbitrary Einstein constants. 
\end{abstract}

\text{} 

\vspace{-5mm}

\maketitle

\vspace{-5mm}

\section*{Introduction}

Perelman successfully introduced Ricci flow with surgery to prove the Poincar\'e conjecture in dimension three and more generally Thurston's geometrization conjecture, \cite{PerelmanEntropyFormula, PerelmanRFwithSurgery}. In particular, Perelman could perform a careful singularity analysis continue the Ricci flow past singularities. 

Work of Feldman-Ilmanen-Knopf \cite{FIKSolitons} and M{\'a}ximo \cite{MaximoBlowUpFourDimRF} shows that in dimension four and higher, Ricci flows on compact manifolds may also develop conical singularities. Moreover, the examples of Feldman-Ilmanen-Knopf \cite{FIKSolitons} and Angenent-Knopf \cite{AngenentKnopfRSConicalSingNonuniqueness} indicate that asymptotically conical expanding Ricci solitons may be used to continue the flow past the singular time. This is also supported by a result of Gianniotis-Schulze  \cite{GianniotisSchulzeRFIsolatedConicalSingularities} who constructed Ricci flows on compact manifolds with conical singularities by gluing in asymptotically conical expanding gradient Ricci solitons.

Based on these examples, Bamler-Chen \cite[Question 1.1]{BamlerChenDegreeTheoryRS} asked if for a given $4$-dimensional Riemannian cone with nonnegative scalar curvature there is a gradient expanding Ricci soliton with nonnegative scalar curvature which is asymptotic to the given cone. Bamler-Chen moreover solved the question affirmatively if the link of the cone is diffeomorphic to the $3$-sphere. 

In this paper we provide examples of expanding gradient Ricci solitons asymptotic to any cone whose link is isometric to $S^2 \times S^1$ where the metrics on the spheres are round with arbitrary radii.

More generally, we prove the following theorem. 

\begin{maintheorem}
\label{MainTheorem}
    Let $d_1 \geq 1$ and let $(M_i,g_i)$ be Einstein manifolds with $\Ric(g_i)=\mu_i g_i$ for $i=2, \ldots, r.$

    \begin{enumerate}
        \item There exists an $r$-parameter family of complete gradient expanding Ricci solitons and an $(r-1)$-parameter family of complete Einstein metrics with negative scalar curvature on $\R^{d_1+1} \times M_2 \times \ldots \times M_r$.
        \item If $\mu_i \geq 0$ for all $i$, then the expanding Ricci solitons are asymptotically conical. 
        \item If $d_1\ge 2$ and $\mu_i=0$ for all $i\ne 1$, then all cones with link of the form
    \begin{align*}
     \left( S^{d_1} \times  M_2 \times \ldots \times M_r, \  \sigma_1^{-2} g_1 + \ldots + \sigma_r^{-2} g_r \right), \ \sigma_i>0
    \end{align*}
    occur as asymptotic cones, where $g_1$ denotes the round metric on $S^{d_1}$.
    \end{enumerate}
\end{maintheorem}

By Bamler-Chen \cite{BamlerChenDegreeTheoryRS}, the expanding Ricci solitons have positive scalar curvature if the asymptotic cone has positive scalar curvature, see also proposition \ref{ScalAtInfinity}. \vspace{2mm}

Note that Theorem \ref{MainTheorem} (a) does not make any assumptions on the Einstein constants. This generalizes work of B\"ohm \cite{BohmNonCompactComhomOneEinstein} on Einstein manifolds with negative scalar curvature and Dancer-Wang \cite{DWExpandingSolitons} on expanding Ricci solitons who all considered Einstein manifolds $(M_i, g_i)$ with positive scalar curvature and $d_1 \geq 2$. The corresponding results for $d_1=1$ were established by Buzano-Dancer-Gallaugher-Wang in \cite{BDGWExpandingSolitons}. All examples are also of warped product type.

Theorem \ref{MainTheorem} (b) confirms the expectation of Dancer-Wang \cite[Remark 3.16]{DWExpandingSolitons} that their Ricci solitons are indeed asymptotically conical. The case of doubly warped products was first considered by Gastel-Kronz \cite{GKExpandingRS} who constructed asymptotically conical expanding Ricci solitons on $\R^{d_1+1} \times M$ with $M$ positive Einstein. Angenent-Knopf \cite{AngenentKnopfRSConicalSingNonuniqueness} gave an independent construction of expanders on $\R^{p+1} \times S^q$ for $p, q \geq 2$ and $p+q \leq 8$ and moreover proved that in this setting there are multiple expanding Ricci solitons asymptotic to the same cone. 

Expanding Ricci solitons with nonnegative respectively positive curvature operators coming out of cones were constructed by Schulze-Simon \cite{SchulzeSimonExpandersComingOutOfCones} and Deruelle \cite{DeruelleSmoothingPosCurvedMetricCones}. In particular, Deruelle provided a classification of asymptotically conical gradient expanding Ricci solitons with nonnegative curvature operators. In contrast, the examples in Theorem \ref{MainTheorem} (c) always have negative Ricci curvatures at the singular orbit in the directions tangent to the Ricci flat factors. 

In the K\"ahler case, generalizing earlier work of Cao \cite{CaoLimitsOfSolutionsKRF}, Dancer-Wang \cite{DWCohomOneSolitons}, Feldman-Ilmanen-Knopf \cite{FIKSolitons} and Siepmann \cite{SiepmannRFofFRcones}, Conlon-Deruelle \cite{ConlonDeruelleExpandingConicalKRS} constructed asymptotically conical expanding gradient K\"ahler Ricci solitons on the total space of the vector bundle $L^{\oplus(n+1)}$, where $L$ is a negative line bundle over a compact K\"ahler manifold $X$ with $c_1( K_X \otimes (L^{*})^{n+1})>0$ and $n \in \N_0.$ Further classification results, in particular for asymptotically conical gradient expanding K\"ahler Ricci solitons of complex dimension two, are established by Conlon-Deruelle-Sun in \cite{CDSclassificationExpShrKRS}. \vspace{2mm}

\textit{Strategy of the proof and structure of the paper.} Recall that an expanding gradient Ricci soliton $(M,g,u)$ is a Riemannian manifold $(M,g)$ together with a smooth function $u$ on $M$ such that
\begin{align*}
    \Ric + \Hess u + \frac{\varepsilon}{2} = 0
\end{align*}
for some constant $\varepsilon>0.$ The Ricci solitons constructed in this paper are multiple warped products on $(0,T) \times S^{d_1} \times M_2 \times \ldots \times M_r$ where the sphere smoothly collapses to a point as $t \to 0.$

In section \ref{SectionMultipleWarpedProductRS} we recall the Ricci soliton equations for multiple warped product manifolds and show completeness of the metrics, i.e. Theorem \ref{MainTheorem} (a), by establishing that the shape operator of the hypersurface $\{ t \} \times S^{d_1} \times M_2 \times \ldots \times M_r$ remains positive definite. This suffices due to \cite[Proposition 1.6]{WinkSolitonsViaEstimatesOnPotential}. The expanding Ricci solitons are parametrized by $(\Bar{f}_2, \ldots, \Bar{f}_r, C)$ where $\Bar{f}_i>0$ rescale the metric of the singular orbit, $\left(M_2 \times \ldots \times M_r, \sum_{i=2}^r \Bar{f}_i^2 g_i \right),$ and $C<0$ corresponds to the second derivative of the soliton potential $u$ at the singular orbit. 

In order to consider the limit trajectory as $C \to - \infty$ in section \ref{SectionSubsystem}, we desingularize the Ricci soliton equations in section \ref{SectionDesingularization} using a well-known coordinate change. Sections \ref{SectionAsymptoticsMetric} and \ref{SectionScalAtInifity} establish the asymptotic behavior of the metrics. In particular, we show that the expanding Ricci solitons are asymptotically conical in lemma \ref{AsymptoticLemma} and thus prove Theorem \ref{MainTheorem} (b). This relies on several ODE comparison results, in particular lemma \ref{LemmaAsymptotics}. We also compute the scalar curvature at infinity of the solitons in proposition \ref{ScalAtInfinity}. 

Section \ref{SectionSubsystem} studies the aforementioned limit trajectory as $C=(d_1+1) \Ddot{u}(0) \to - \infty$ which correspond to an invariant subsystem of the Ricci soliton ODE. One motivation to consider this limit is that the scalar curvature satisfies $R=-C-\varepsilon u - \Dot{u}^2 -(n+1) \frac{\varepsilon}{2}$ and thus the (non-geometric) limit corresponds (intuitively) to having infinite scalar curvature. Moreover, if the singular orbit consists only of Ricci flat manifolds, then the cone angle of the sphere factor must become arbitrarily small as $C \to - \infty$ to obtain arbitrarily large scalar curvature. The asymptotic behavior of the trajectory of the rescaled limit system quantifies this idea. 

In the final section \ref{SectionProofMainTheorem} we use this to show that there are also regular trajectories, i.e. trajectories corresponding to complete asymptotically conical expanding Ricci solitons, where the sphere factor of the link is (and remains) arbitrarily small. By considering the limit $C \to 0$, i.e. the Einstein trajectories of Theorem \ref{MainTheorem} (a), we show that similarly there are expanding Ricci solitons with spheres of arbitrarily large radii in the link. Combining these observations we show that we can find an aysmptotically conical expanding Ricci soliton with any given cone angle of the sphere factor in lemma \ref{Simga1Realized}. Theorem \ref{MainTheorem} (c) follows by suitably rescaling the Ricci flat factors of the link, i.e. by choosing $\Bar{f}_i >0$ appropriately. \vspace{2mm}

\textit{Acknowledgements.} MW thanks Eric Chen and Tristan Ozuch for bringing his attention to the problem of constructing expanding Ricci solitons on $\R^3 \times S^1$ with positive scalar curvature.

JN acknowledges support by the Alexander von Humboldt Foundation through Gustav Holzegel's Alexander von Humboldt Professorship endowed by the Federal Ministry of Education and Research. Both authors are funded by the Deutsche Forschungsgemeinschaft (DFG, German Research Foundation) under Germany's Excellence Strategy EXC 2044–390685587, Mathematics M\"unster: Dynamics–Geometry–Structure.

\section{Multiple warped product gradient Ricci solitons}
\label{SectionMultipleWarpedProductRS}

For $r \geq 2$ let $(M_i,g_i)_{i=2, \ldots, r}$ be Einstein manifolds with $\Ric(g_i)=\mu_i g_i$. Let $d_1 \geq 1$ and set $d_i = \dim M_i$ for $i=2, \ldots, r.$ On $(0,T) \times S^{d_1} \times M_2 \times \ldots \times M_r$
consider the metric
\begin{align*}
    dt^2 + \sum_{i=1}^r f_i^2(t) g_i,
\end{align*}
where $g_1$ denotes the round metric on $S^{d_1}$ of radius $1,$ so that $\mu_1=d_1-1.$ Note that the shape operator and the Ricci curvature of the hypersurface $\{ t \} \times S^{d_1} \times M_2 \times \ldots \times M_r$ are given by
\begin{align*}
    L_t  = \diag \left( \frac{\Dot{f}_1}{f_1} \id_{d_1}, \ldots, \frac{\Dot{f}_r}{f_r} \id_{d_r} \right) \ \text{ and } \
    r_t  = \diag \left( \frac{\mu_1}{f_1^2} \id_{d_1}, \ldots, \frac{\mu_r}{f_r^2} \id_{d_r} \right).
\end{align*}
In this case the gradient Ricci soliton equation with soliton potential $u=u(t)$ reduces to 
\begin{align*}
    \frac{d}{dt}  (- \Dot{u} + \tr(L)) & = - \tr(L^2) + \frac{\varepsilon}{2}, \\
    \frac{d}{dt}  L & = - (- \Dot{u} + \tr(L)) L + r + \frac{\varepsilon}{2} \id.
\end{align*}
Furthermore, the Ricci soliton satisfies Hamilton's \cite{HamiltonFormationOfSingularitiesRF} conservation law
\begin{align*}
    \Ddot{u} + (- \Dot{u} \tr(L) ) \Dot{u} = \tr(L^2) + \tr(r) + (n-1) \frac{\varepsilon}{2} - (- \Dot{u} + \tr(L))^2 =  C + \varepsilon u,
\end{align*}
where $n=\sum_{i=1}^r d_i$ is the dimension of the hypersurface.

By the work of Buzano \cite{BuzanoInitialValueGRS}, the metric extends smoothly to a Ricci soliton metric on $\R^{d_1+1} \times M_2 \times \ldots \times M_r$ if the warping functions $f_i$ satisfy the boundary conditions
\begin{align*}
    f_1(0)=1, \ \Dot{f}_1(0)=0, \ f_i(0) = \Bar{f}_i > 0 \ \text{ and } \ \Dot{f}_i(0) = 0
\end{align*}
for $i=2, \ldots, r.$ Furthermore, if we fix $u(0)=0,$ then the smoothness condition for the soliton potential is
\begin{align*}
    u(0)=0, \ \Dot{u}(0)=0, \ \Ddot{u}(0) = \frac{C}{d_1+1}.
\end{align*}
With this choice of initial conditions, $C$ is the same constant as in the conservation law. 

The ambient scalar curvature of the warped product is given by 
\begin{equation*}
    R = \tr(r) - \tr(L^2) - \tr(L)^2 - 2 \tr( \dot{L} ).
\end{equation*}
For Ricci solitons we thus obtain
\begin{align*}
    R & = - C - \varepsilon u - \dot{u}^2 - (n+1) \frac{\varepsilon}{2} \\
    & = - \tr(r) - \tr(L^2) + \tr(L)^2 - 2 \left( \Dot{u} \tr(L) + n \frac{\varepsilon}{2} \right)
\end{align*}
and in particular $R(0)=-C- (n+1) \frac{\varepsilon}{2}.$ By the work of Chen \cite{ChenStrongUniquenessRF}, a complete, non-Einstein expanding Ricci soliton satisfies $R+(n+1) \frac{\varepsilon}{2} > 0.$ In particular, we require $C<0$ as a necessary condition to obtain complete expanding Ricci solitons. Furthermore, Buzano-Dancer-Gallaugher-Wang \cite{BDGWExpandingSolitons} observed that if $C<0,$ then the conservation law implies that $u(t)<0,$ $\Dot{u}(t)<0$ and $\Ddot{u}(t)<0$ for $t>0,$ see also \cite[Proposition 1.2]{WinkSolitonsViaEstimatesOnPotential}. Furthermore, $C=0$ corresponds to Einstein metrics.

Note that the components of the shape operator satisfy
\begin{align*}
    \frac{d}{dt} \frac{\Dot{f}_i}{f_i} = - ( - \Dot{u} + \tr(L)) \frac{\Dot{f}_i}{f_i} + \frac{\mu_i}{f_i^2} + \frac{\varepsilon}{2}
\end{align*}
and thus $(d_1+1) \Ddot{f}_i(0) = \frac{\mu_i}{\Bar{f}_i} + \frac{\varepsilon}{2} \Bar{f}_i.$ 

For $i\geq 2$ choose $\Bar{f}_i >0$ such that $\frac{\mu_i}{\Bar{f}_i^2} + \frac{\varepsilon}{2}>0$. Then $\Ddot{f}_i(0)>0$ and therefore $\Dot{f}_i(t)>0$ for small $t>0.$ Note that $\Dot{f}_i(t)>0$ is preserved as long as $\frac{d}{dt} \frac{\Dot{f}_i}{f_i} \geq - ( - \Dot{u} + \tr(L)) \frac{\Dot{f}_i}{f_i}.$ For $\mu_i \geq 0$ this is immediate and for $\mu_i<0$ note that $\Dot{f}_i(t)>0$ also implies $\frac{\mu_i}{f_i(t)^2} + \frac{\varepsilon}{2} > \frac{\mu_i}{\Bar{f}_i^2} + \frac{\varepsilon}{2}>0.$ Therefore the shape operator remains positive definite and completeness of the metric follows as in \cite[Proposition 1.6]{WinkSolitonsViaEstimatesOnPotential}. This proves part (a) of Theorem \ref{MainTheorem}. \vspace{2mm}

In the following we will therefore only consider trajectories with 
\begin{align}
\label{RefinedInitialCondition}
    \Bar{F}_i= \frac{\mu_i}{\Bar{f}_i^2} + \frac{\varepsilon}{2}>0
\end{align}
for $i=2, \ldots, r.$ By completeness of the metric, these trajectories are defined for $t \in [0, \infty).$

\begin{remark}
\label{ProductFactorExpanders}
    \normalfont
    Note that for $\Bar{F}_i=0$ the warping function $f_i$ remains constant and $M_i$ splits off as a product factor.
\end{remark}

\begin{remark}
    \normalfont
    The above construction also applies to Ricci flat metrics and steady Ricci solitons provided the Einstein manifolds $(M_i,g_i)$ have positive scalar curvature. This recovers metrics constructed by B\"ohm \cite{BohmNonCompactComhomOneEinstein}, Dancer-Wang \cite{DWSteadySolitons} and Buzano-Dancer-Wang \cite{BDWSteadySolitons}. In the above approach it is also possible to include Ricci flat manifolds $(M_i,g_i).$ However, as in remark \ref{ProductFactorExpanders}, these split off as product factors. In fact, if $(M,g)$ is a steady Ricci soliton with potential $u$ and $(N,h)$ is Ricci flat, then the Riemannian product $(M \times N, g + h)$ is a steady Ricci soltion with soliton potential $u \circ \pi_M.$
\end{remark}

\section{Desingularization of the Ricci Soliton ODE}
\label{SectionDesingularization}

Set
\begin{align*}
    \mathcal{L} = \frac{1}{- \dot{u}+\tr(L)},\ \
    \frac{d}{ds} = \mathcal{L} \frac{d}{dt}, \ \
    X_i = \mathcal{L} \frac{\dot{f}_i}{f_i} \ \text{ and } \
    Y_i = \frac{\mathcal{L}}{f_i}
\end{align*}
and denote differentiation with respect to $s$ by ${}^{'}.$ Then, 
\begin{align*}
    \mathcal{L}^{'} & = \mathcal{L} \left( \sum_{j=1}^r d_j X_j^2 - \frac{\varepsilon}{2} \mathcal{L}^2 \right), \\
    X_i^{'} & = X_i \left( \sum_{j=1}^r d_j X_j^2 - \frac{\varepsilon}{2} \mathcal{L}^2 - 1 \right) + \mu_i Y_i^2 + \frac{\varepsilon}{2} \mathcal{L}^2, \\
    Y_i^{'} & = Y_i \left( \sum_{j=1}^r d_j X_j^2 - \frac{\varepsilon}{2} \mathcal{L}^2 - X_i \right)
\end{align*}
for $i=1, \ldots, r.$ Furthermore, let 
\begin{align*}
    \mathcal{S}_1 & = \sum_{i=1}^r d_i X_i^2 + \sum_{i=1}^r d_i \mu_i Y_i^2 + (n-1) \frac{\varepsilon}{2} \mathcal{L}^2 - 1, \\
    \mathcal{S}_2 & = \sum_{i=1}^r d_i X_i -1
\end{align*}
and observe that 
\begin{align*}
    \frac{1}{2} \mathcal{S}_1^{'} & = \left( \sum_{i=1}^r d_i X_i^2 - \frac{\varepsilon}{2} \mathcal{L}^2 \right) \mathcal{S}_1 + \frac{\varepsilon}{2} \mathcal{L}^2 \mathcal{S}_2, \\
    \mathcal{S}_2^{'} & = \mathcal{S}_1 +  \left( \sum_{i=1}^r d_i X_i^2 - \frac{\varepsilon}{2} \mathcal{L}^2 -1 \right) \mathcal{S}_2.
\end{align*}

Note that, by the conservation law, $\mathcal{S}_1 = ( C + \varepsilon u) \mathcal{L}^2$ and $\mathcal{S}_2 = \dot{u} \mathcal{L}$. In fact, $\frac{d}{ds} \frac{\mathcal{S}_1}{\mathcal{L}^2} = \varepsilon \mathcal{S}_2$ and $\mathcal{S}_1 - \mathcal{S}_2 = \Ddot{u} \mathcal{L}^2.$

Solutions to the Ricci soliton equation satisfying the smoothness conditions of section \ref{SectionMultipleWarpedProductRS} correspond to trajectories in the unstable manifold of the stationary point
\begin{align}
\label{RescaledSmoothnessConditions}
    \mathcal{L}=0, \ X_1=Y_1=\frac{1}{d_1}, \ X_i=Y_i=0
\end{align}
for $i=2, \ldots, r.$ 

In the following we are only going to be interested in trajectories that are induced by solutions satisfying the initial conditions \eqref{RefinedInitialCondition} and $C \leq 0.$ Note that, due the conservation law, these trajectories are contained in the locus $\{ \mathcal{S}_1 < 0 \} \cap \{ \mathcal{S}_2 < 0 \}$ for $C<0$ and in the locus $\{ \mathcal{S}_1 = 0 \} \cap \{ \mathcal{S}_2 = 0 \}$ for $C=0.$ Furthermore they satisfy $\mathcal{L}, X_i, Y_i > 0$ as well as $\lim_{s \to - \infty} \mu_i \frac{Y_i^2}{\mathcal{L}^2} + \frac{\varepsilon}{2} = \Bar{F}_i >0.$ 

If $\mu_i<0,$ then $\frac{d}{ds} \frac{Y_i}{\mathcal{L}} = - \frac{Y_i}{\mathcal{L}} X_i \leq 0$ implies that $-\mu_i Y_i^2 < \left( \frac{\varepsilon}{2} - \Bar{F}_i \right) \mathcal{L}^2$ and the conservation law $\mathcal{S}_1 \leq 0$ shows that $\mathcal{L}$, all $Y_i$ with $\mu_i \neq 0,$ and all $X_i$ are bounded. Furthermore, for all $i,$ the ODE for $Y_i$ thus shows that $Y_i$ cannot blow up in finite time. In particular, solutions are defined for all $s \in \R.$ \vspace{2mm}

To obtain a formula for the scalar curvature $R$ in the new coordinate system we set 
\begin{align*}
    \mathcal{R} = R \mathcal{L}^2.
\end{align*}
From $R =  - \tr(r) - \tr(L^2) + \tr(L)^2 - 2 \left( \Dot{u} \tr(L) + n \frac{\varepsilon}{2} \right)$ it follows that 
\begin{align*}
    \mathcal{R} 
    & \ = 2 \sum_{i=1}^r d_i X_i - \sum_{i=1}^r d_i X_i^2 - \left( \sum_{i=1}^r d_i X_i \right)^2 - \sum_{i=1}^r d_i \mu_i Y_i^2 - n \varepsilon \mathcal{L}^2 \\
    & \ = - \left( \mathcal{S}_1 + \mathcal{S}_2^2 + (n+1) \frac{\varepsilon}{2} \mathcal{L}^2 \right)
\end{align*}
and it is straightforward to compute that
\begin{align*}
    \frac{1}{2} \frac{d}{ds} \mathcal{R}= \left( \sum_{i=1}^r d_i X_i^2 - \frac{\varepsilon}{2} \mathcal{L}^2 \right) \mathcal{R} + \left( \mathcal{S}_2 - \mathcal{S}_1 - \frac{\varepsilon}{2} \mathcal{L}^2 \right) \mathcal{S}_2.
\end{align*}

For applications in section \ref{SectionScalAtInifity} we note the following. 

\begin{proposition}
\label{OriginAttractor}
   Suppose that $\mu_i \geq 0.$ Then the origin is a stable attractor. 
\end{proposition}
\begin{proof}
    The linearization of the Ricci soliton ODE at the origin shows that all eigenvalues are nonpositive and there is a center manifold. As in \cite[Proof of Proposition 3.11]{DWExpandingSolitons} one shows that $\sum_{i=1}^r Y_i^2 + \frac{\varepsilon}{2} \mathcal{L}^2$ is a Lyapunov function for the flow on the center manifold near the origin. In particular, the origin is a sink for the flow, see also \cite[Theorem 2]{CarrCenterManifoldTheory}.
\end{proof}

\section{Approximate asymptotics of the metrics}
\label{SectionAsymptoticsMetric}

In the case of Einstein metrics with negative scalar curvature, the condition $\sum_{i=1}^r d_i X_i =1$ implies $\sum_{i=1}^r d_i X_i^2 \geq 1/n$ and the ODE for $\mathcal{L}$ shows that $\mathcal{L}$ is bounded away from zero for $s>s_0.$ As $\mathcal{L}$ is moreover bounded and $Y_i / \mathcal{L}$ is decreasing, all $Y_i$ are hence bounded. Thus, the $\omega$-limit set of the ODE is connected, compact, non-empty and invariant under the flow. Furthermore, $\frac{d}{ds} \frac{Y_i}{\mathcal{L}} = - \frac{Y_i}{\mathcal{L}} X_i$ also implies that $X_i \cdot Y_i = 0$ on the $\omega$-limit set. However, the ODE for $X_i$ shows that $X_i=0$ is impossible on the $\omega$-limit set since $\mu_i Y_i^2 + \frac{\varepsilon}{2} \mathcal{L}^2$ is bounded away from zero for $s > s_0.$ For $\mu_i<0$ this follows by the choice of $\Bar{f}_i>0$ in \eqref{RefinedInitialCondition}. Therefore, $Y_i \to 0$ as $s \to \infty.$

As $\sum_{j=1}^r d_j X_j^2 - \frac{\varepsilon}{2}\mathcal{L}^2 -1$ is negative and bounded away from zero for $s > s_0$, the ODE
\begin{align*}
    (X_k-X_l)^{'} = (X_k-X_l) \left( \sum_{j=1}^r d_j X_j^2 - \frac{\varepsilon}{2}\mathcal{L}^2 -1 \right) + \mu_k Y_k^2 - \mu_l Y_l^2
\end{align*}
shows by comparison that $X_k-X_l \to 0$ as $s \to \infty$ for all $k,l$. Thus $X_i \to  \frac{1}{n}$ and then also $\frac{\varepsilon}{2}\mathcal{L}^2 \to \frac{1}{n}$ as $s \to \infty.$ In particular, $\frac{\Dot{f}_i}{f_i} = \frac{X_i}{\mathcal{L}} \to \sqrt{n\frac{\varepsilon}{2}}$ as $t \to \infty.$ \vspace{2mm}

In the case of expanding Ricci solitons, the work of Buzano-Dancer-Gallaugher-Wang \cite{BDGWExpandingSolitons} shows that $-\frac{\Dot{u}}{t} \to \frac{\varepsilon}{2}$ as $t \to \infty$ and moreover there is $t_0>0$ such that $| \tr(L) | < \sqrt{n \frac{\varepsilon}{2}}$ for $t>t_0.$ This implies that $\mathcal{L}, X_i, Y_i \to 0$ as $s \to \infty.$ As $X_i \geq 0$ and $\mu_i \frac{Y_i^2}{\mathcal{L}^2}+\frac{\varepsilon}{2} >0$ is bounded away from zero, one can now proceed as in \cite{DWExpandingSolitons}, see also remark \ref{AsymptoticsXL2YL}, to deduce that $\frac{X_i}{\mathcal{L}^2} \to \frac{\varepsilon}{2},$ $\frac{Y_i}{\mathcal{L}} \to 0$ as $s \to \infty$ and to conclude that $\frac{\varepsilon}{2} \mathcal{L} \cdot t \to 1,$ $\frac{\Dot{f}_i}{f_i} \cdot t \to  1$ as $t \to \infty.$

\section{Scalar curvature at infinity}
\label{SectionScalAtInifity}

From now on we restrict to expanding Ricci solitons on $\R^{d_1+1} \times M_2 \times \ldots \times M_r$ where $d_1 \geq 2$ and $(M_i,g_i)$ are Einstein with nonnegative scalar curvature for $i=2, \ldots, r$. In particular, 
\begin{align*}
    \mu_1 = d_1 -1>0 \ \text{ and } \ \mu_i \geq 0  \text{ for }  i=2, \ldots, r.
\end{align*}

To distinguish between trajectories corresponding to expanding Ricci solitons and Einstein metrics we make the following definition. 

\begin{definition}
A trajectory in the unstable manifold of \eqref{RescaledSmoothnessConditions} with $\mathcal{L}>0$, $X_i>0,$ $Y_i>0$ for $i=1, \ldots, r$ is called {\em regular} if $\mathcal{S}_1<0$ and $\mathcal{S}_2<0$ respectively {\em Einstein} if $\mathcal{S}_1=\mathcal{S}_2=0.$
\end{definition}
Recall from section \ref{SectionDesingularization} that the loci $\{ \mathcal{S}_1 < 0 \} \cap \{ \mathcal{S}_2 < 0 \}$ and $\{ \mathcal{S}_1 = 0 \} \cap \{ \mathcal{S}_2 = 0 \}$ are preserved by the ODE as $\varepsilon > 0.$

\begin{remark}
\normalfont
    The trajectory corresponding to $(\Bar{f}_2, \ldots, \Bar{f}_r, C)$ is regular if and only if $\Bar{f}_i >0$ for $i=2, \ldots, r$ and $C < 0$.

    Indeed, in section \ref{SectionDesingularization} we established that trajectories induced by the smoothness conditions of section \ref{SectionMultipleWarpedProductRS} are regular for $\Bar{f}_i>0$ and $C<0$, or Einstein if $\Bar{f}_i>0$ and $C=0.$

    The converse follows from the linearization at \eqref{RescaledSmoothnessConditions}, cf. \cite[Section 2]{DWExpandingSolitons}. Note that for $d_1 \geq 2$ the fixed point \eqref{RescaledSmoothnessConditions} is hyperbolic and that we recover $\Bar{f}_i=f_i(0)$ and $C=(d_1+1) \Ddot{u}(0)$ via
    \begin{align*}
        \Bar{f}_i = \lim_{s \to - \infty} \frac{Y_i}{\mathcal{L}} \ \text{ and } C= \lim_{s \to - \infty} \frac{\mathcal{S}_1}{\mathcal{L}^2}.
    \end{align*}
\end{remark}

To study the precise asymptotic behavior of the trajectories, we set
\begin{definition}
    For $i=1, \ldots, r$ define $\sigma_i$ by
    \begin{align*}
        \sigma_i^{-1} = \lim_{t \to \infty} \Dot{f}_i = \lim_{s \to \infty} \frac{X_i}{Y_i},
    \end{align*}
where by convention we set $\sigma_i=0$ if the limit diverges to infinity.
\end{definition}

Along Einstein trajectories we have $\sigma_i=0$ according to section \ref{SectionAsymptoticsMetric}. To prove that the $\sigma_i$ are well-defined also along regular trajectories we use following lemma. 
\begin{lemma}
\label{LemmaAsymptotics}
    For $i=1,2$ let $c_i \colon \R \to \R$ be smooth functions and suppose that $c_i(s) \to c_i^{*} > 0$ as $s \to \infty.$ Suppose that $f \colon \R \to \R$ satisfies $ f{'} = - c_1 f + c_2.$ 
    
    Then $f$ converges to $\frac{c_2^{*}}{c_1^{*}}$ as $s \to \infty.$
\end{lemma}
The lemma follows as in \cite[Lemma 3.13]{DWExpandingSolitons}. Recall from section \ref{SectionAsymptoticsMetric} that regular trajectories satisfy  
\begin{align*}
 \mathcal{L}, X_i, Y_i \to 0, \ \frac{X_i}{\mathcal{L}^2} \to \frac{\varepsilon}{2} \ \text{ and } \ \frac{Y_i}{\mathcal{L}} \to 0
\end{align*}
for $i=1, \ldots, r$ as $s \to \infty.$ 

\begin{remark}
\label{AsymptoticsXL2YL}
    \normalfont
In fact, following \cite{DWExpandingSolitons}, in order to prove the last two statements, one first observes that $\frac{d}{ds} \frac{Y_i}{\mathcal{L}} = - \frac{Y_i}{\mathcal{L}} X_i \leq 0$ and thus $\frac{Y_i}{\mathcal{L}}$ converges. Then one applies lemma \ref{LemmaAsymptotics} to 
\begin{align*}
    \frac{d}{ds} \frac{X_i}{\mathcal{L}^2} = \frac{X_i}{\mathcal{L}^2} \left( - \sum_{j=1}^r d_j X_j^2 + \frac{\varepsilon}{2}\mathcal{L}^2 - 1 \right) + \mu_i \frac{Y_i^2}{\mathcal{L}^2} + \frac{\varepsilon}{2}
\end{align*}
to deduce convergence of $\frac{X_i}{\mathcal{L}^2}$. Finally one establishes $\frac{Y_i}{\mathcal{L}} \to 0$ by considering $\frac{Y_i'}{\mathcal{L}'}$ and invoking L'H\^opital's rule. This implies $\frac{X_i}{\mathcal{L}^2} \to \frac{\varepsilon}{2}.$
\end{remark}

\begin{lemma}
    \label{AsymptoticLemma}
    Along any regular trajectory, $\sigma_i$ exist and $\sigma_i \in (0,\infty)$ for $i=1, \ldots, r.$ Moreover, we have the refined asymptotics
    \begin{align*}
        \frac{X_i-\frac{\varepsilon}{2}\mathcal{L}^2}{\mathcal{L}^4} \to (\frac{\varepsilon}{2})^2 \left(\mu_i \sigma_i^2 + 1 \right)
    \end{align*}
    for $i=1, \ldots, r$ as $s \to \infty.$
\end{lemma}

\begin{proof}
It is straightforward to compute that 
\begin{align*}
    \frac{d}{ds} \frac{X_i- \frac{\varepsilon}{2} \mathcal{L}^2}{\mathcal{L}^4} = & \  \frac{X_i - \frac{\varepsilon}{2} \mathcal{L}^2}{\mathcal{L}^4} \left( -3 \sum_{i=1}^r d_j X_j^2 + 3 \frac{\varepsilon}{2} \mathcal{L}^2 - 1 \right) + \mu_i \left( \frac{Y_i}{\mathcal{L}^2} \right)^2 \\
    & \ + \left( \frac{\varepsilon}{2} \right) \left( \frac{\varepsilon}{2} - \sum_{i=1}^r d_j \left( \frac{X_j}{\mathcal{L}} \right)^2 \right).
\end{align*}

Along regular trajectories we have $\frac{X_i}{\mathcal{L}^2} \to \frac{\varepsilon}{2}$ as $s \to \infty.$ Thus $\frac{X_i}{\mathcal{L}} \to 0$ as $s \to \infty$ and since $\mu_i \geq 0$ it follows that $X_i>\frac{\varepsilon}{2} \mathcal{L}^2$ for large times. In fact, if $\sigma_i$ converges, then $\frac{Y_i}{\mathcal{L}^2}=\frac{X_i}{\mathcal{L}^2} \frac{Y_i}{X_i}  \to \frac{\varepsilon}{2}\sigma_i$ as $s \to \infty$ and lemma \ref{LemmaAsymptotics} implies that $\frac{X_i - \frac{\varepsilon}{2} \mathcal{L}^2}{ \mathcal{L}^4 } \to (\frac{\varepsilon}{2})^2 \left( \mu_i \sigma_i^2 + 1 \right)$ as $s \to \infty$.

To show that $\sigma_i$ exist and $\sigma_i \in (0, \infty)$ it suffices to establish convergence of $\frac{Y_i}{\mathcal{L}^2}$ in $(0, \infty).$ Note that 
    \begin{align*}
        \frac{d}{ds} \frac{Y_i}{\mathcal{L}^2} =  \frac{Y_i}{\mathcal{L}^2} \left( - \sum_{i=1}^2 d_j X_j^2 + \frac{\varepsilon}{2} \mathcal{L}^2 - X_i \right) <0
    \end{align*}
for large times as eventually we have $X_i > \frac{\varepsilon}{2} \mathcal{L}^2$ for $i=1, \ldots, r.$

It remains to show that $\sigma_i$ are positive. Pick some fixed time $T>0$. By the previously established asymptotic estimates, there are constants $c_i$ such that 
    \begin{align*}
        \left| - \sum_{i=1}^r d_j X_j^2 + \frac{\varepsilon}{2} \mathcal{L}^2 - X_i \right| \le c_i \mathcal{L}^4
    \end{align*} for $t\geq T$.

    Using $\frac{d}{ds}= \mathcal{L}\frac{d}{dt}$ as well as $\frac{\varepsilon}{2} \mathcal{L} \cdot t \to 1$ for $t \to \infty$, we obtain

    \begin{align}
        \label{eqAsBounds}
        \left| \frac{d}{dt}\frac{Y_i}{\mathcal{L}^2} \right| \le \Tilde{c}_i \frac{Y_i}{\mathcal{L}^2} \frac{1}{t^3}
    \end{align}
for $t \geq T$ and some constants $\Tilde{c}_i>0.$

    The corresponding differential equation 
    \begin{align*}
        \frac{dg}{dt} = - \Tilde{c} \frac{g}{t^3}, \ g(T)=g_0>0
    \end{align*}
    has the explicit solution  $g(t)=g_0 \exp \left( \frac{\Tilde{c}}{2} \left( \frac{1}{t^2} - \frac{1}{T^2} \right) \right).$ In particular, $g(t)$ converges to a positive constant as $t \to \infty.$ By comparison, $\frac{Y_i}{\mathcal{L}^2}$ remains bounded away from zero for $t \geq T$ and we obtain positivity of $\sigma_i$.
\end{proof}

\begin{proposition}
    \label{ScalAtInfinity}
    Along any regular trajectory the scalar curvature satisfies
    \begin{align*}
        \frac{R}{\mathcal{L}^2} \to \left( \frac{\varepsilon}{2} \right)^2 \left( \sum_{i=1}^r d_i \mu_i \sigma_i^2 - n(n-1) \right)
    \end{align*}
    as $s \to \infty.$

    In particular, the induced gradient expanding Ricci soliton has positive scalar curvature if its asymptotic cone 
    \begin{align*}
     \left( (0, \infty) \times S^d \times M_2 \times \ldots \times M_r, \ dt^2 + (\sigma_1^{-1} t)^2 g_{S^d} + \sum_{i=2}^r ( \sigma_i^{-1} t)^2 g_{M_i} \right)   
    \end{align*}
    has positive scalar curvature. 
\end{proposition}
\begin{proof}
    The formula for the scalar curvature $\mathcal{R}=R\mathcal{L}^2$ from section \ref{SectionDesingularization} shows that 
\begin{align*}
     \frac{\mathcal{R}}{\mathcal{L}^4}  = & \ 2 \sum_{i=1}^r d_i \frac{X_i-\frac{\varepsilon}{2} \mathcal{L}^2}{\mathcal{L}^4} - \sum_{i=r}^r d_i \left( \frac{X_i}{\mathcal{L}^2} \right)^2 - \left( \sum_{i=1}^r d_i \frac{X_i}{\mathcal{L}^2} \right)^2 - \sum_{i=1}^r d_i \mu_i \left( \frac{Y_i}{\mathcal{L}^2} \right)^2 \\
     \to & \ \left( \frac{\varepsilon}{2} \right)^2 \left( 2 \sum_{i=1}^r d_i (\mu_i \sigma_i^2 +1)  - n - n^2 - \sum_{i=1}^r d_i \mu_i \sigma_i^2 \right)  \\
     = & \ \left( \frac{\varepsilon}{2} \right)^2 \left( \sum_{i=1}^r d_i \mu_i \sigma_i^2 - n(n-1) \right)
\end{align*}
as $s \to \infty.$

We note that the solitons have bounded curvature by direct computation. Therefore we can apply the work of Bamler-Chen \cite[Theorem 2.3]{BamlerChenDegreeTheoryRS} which shows that if the scalar curvature of an expanding Ricci soliton is positive at infinity, then it is positive everywhere. The claim follows from the observation that the scalar curvature of the cone is 
    \begin{align*}
        R_{\text{cone}}=  \left( \sum_{i=1}^r d_i \mu_i \sigma_i^2- n(n-1) \right) \frac{1}{t^2}.
    \end{align*}
\end{proof}

\section{A limiting subsystem}
\label{SectionSubsystem}
Consider the Ricci soliton ODE for $\mathcal{L}=0$ and $X_i=Y_i=0$ for $i=2, \ldots, r$. Then one obtains an invariant subsystem in $X=X_1, Y=Y_1$ with $d=d_1 \geq 2$. Namely,\begin{align*}
    X^{'} & = X(dX^2-1) +(d-1)Y^2, \\
    Y^{'} & = XY ( dX - 1 ).
\end{align*}

The ODE has the fixed points $(X,Y)=(0,0),$ $( \pm \frac{1}{\sqrt{d}},0),$ $(\frac{1}{d}, \pm \frac{1}{d} ).$ The fixed point $(X,Y)=(\frac{1}{d}, \frac{1}{d})$ is a saddle and there is precisely one trajectory in the unstable manifold that emanates into the region $0 < X, Y < \frac{1}{d}.$

\begin{proposition}
\label{AsymptoticsSubsystem}
    Suppose that $0 < X, Y < \frac{1}{d}$ at $s_0 \in \R.$ Then $0 < X, Y < \frac{1}{d}$ for all times and moreover
    \begin{align*}
        X, Y \to 0, \ 
        \frac{X}{Y^2} \to (d-1),  \ \frac{1}{X} \left( \frac{X}{Y^2}-(d-1) \right) \to 2 (d-1)
    \end{align*}
    as $s \to \infty.$
\end{proposition}
\begin{proof}
    Clearly, $ X \geq 0$ and $Y \geq 0$ are preserved.  Furthermore, note that for $0 < X, Y < \frac{1}{d}$ we have $X^{'}, Y{'} < 0$. Thus $0 \leq X, Y \leq \frac{1}{d}$ is also preserved and moreover if $0<X,Y< \frac{1}{d}$ at $s_0 \in \R,$ then $X, Y \to 0$ as $s \to \infty.$

    Furthermore, notice that 
    \begin{align*}
    \frac{d}{ds} \frac{X}{Y^2} = \frac{X}{Y^2} ( -d X^2 + 2X -1 ) + d-1
    \end{align*}
    and thus lemma \ref{LemmaAsymptotics} implies that $\frac{X}{Y^2} \to d-1$ as $s \to \infty$. Similarly, it follows from 
    \begin{align*}
    \frac{d}{ds} \frac{\frac{X}{Y^2} - (d-1)}{Y^2} =  \frac{\frac{X}{Y^2} - (d-1)}{Y^2} \left( - d X^2 + 2X -1 -2 X (dX-1) \right) + (d-1) \frac{X}{Y^2}( 2 - dX )
    \end{align*}
    that $\frac{1}{Y^2} \left( \frac{X}{Y^2} - (d-1) \right) \to 2 (d-1)^2$ as $s \to \infty.$
\end{proof}

Consequently, for trajectories as in proposition \ref{AsymptoticsSubsystem}, there are $\alpha_1 >0$ and $s_1 \in \R$ such that $\left| X - (d-1) Y^2 \right| 
    \leq \alpha_1 X^2$
for $s > s_1.$ Thus there are $\alpha_2 >0$ and $s_2 \geq s_1$ such that the renormalized scalar curvature $\mathcal{R}$ restricted to the subsystem satisfies
\begin{align*}
    \mathcal{R}  = 2d X - d(d+1) X^2 - (d-1) d Y^2 \geq dX - \alpha_2 X^2 >0
\end{align*}
for $s > s_2.$

\section{Expanding Ricci solitons asymptotic to cones}
\label{SectionProofMainTheorem}

Throughout this section we assume that $d_1 \geq 2$ and that the manifolds $(M_i,g_i)$ are Ricci flat, hence $\mu_i=0$ for $i=2, \ldots, r.$

Recall from lemma \ref{AsymptoticLemma} that all $\sigma_i^{-1} = \lim_{t \to \infty} \Dot{f}_i$ exist and are positive. The proof of Theorem \ref{MainTheorem} (c) relies on the following lemma. 

\begin{lemma}
    \label{Simga1Realized}
    For any $\Bar{f}_2, \ldots, \Bar{f}_r>0$ the trajectories corresponding to the initial conditions $(\Bar{f}_2, \ldots, \Bar{f}_r, C)$ achieve all values of $\sigma_1 \in (0, \infty)$ as $C$ varies in $(- \infty,0)$. 
\end{lemma}

Before proving the lemma, we show how it implies Theorem \ref{MainTheorem} (c).

\begin{proof}[Proof of Theorem \ref{MainTheorem} (c)]
    For a given cone 
    \begin{align*}
     \left( (0, \infty) \times S^{d_1} \times M_2 \times \ldots \times M_r, \ dt^2 + (\sigma_1^{-1} t)^2 g_{S^{d_1}} + \sum_{i=2}^r (\sigma_i^{-1} t)^2 g_{i} \right)   
    \end{align*}
    pick $C<0$ according to lemma \ref{Simga1Realized} such that the trajectory with initial condition $f_i(0)=\Bar{f}_i>0$ for $i=2, \ldots, r$ and $\Ddot{u}(0) = \frac{C}{d_1+1}$ satisfies $\lim_{t \to \infty} \Dot{f}_1(t) = \sigma_1^{-1}.$

    Since all $M_i$ are by assumption Ricci flat, we have $\mu_i=0$ for $i=2, \ldots, r$ and we may observe that $Y_i$ decouples from the other equations in the Ricci soliton ODE. In particular, given any regular trajectory $(X_1, Y_1, X_i, Y_i, \mathcal{L})$, we may obtain for any $c_i>0$ another regular trajectory given by $(X_1, Y_1, X_i, c_i \cdot Y_i, \mathcal{L})$. Clearly the limit of $\frac{Y_1}{X_1}$ is not affected by this process while that of $\frac{Y_i}{X_i}$ is rescaled by a factor of $c_i$ for $i=2, \ldots, r$. Since we know this limit to be nonzero along any regular trajectory, we may thus obtain any limits $\sigma_i$ by suitable rescalings of $Y_i$.
\end{proof}

\begin{remark}
    \normalfont
    If the original trajectory has the initial condition $(\Bar{f_2}, \ldots, \Bar{f}_r, C_0)$, then the rescaled trajectory has the initial condition $(c_2^{-1} \cdot \Bar{f_2}, \ldots, c_r^{-1} \cdot \Bar{f}_r, C_0)$.
\end{remark}

The proof of lemma \ref{Simga1Realized} relies on the following observations. 

\begin{proposition}
\label{PreservedInequalities}
    Along regular and Einstein trajectories the following hold for all time:
    \begin{itemize}
        \item $X_1 > X_i$ for $i=2, \ldots, r,$
        \item $X_1 > \sum_{j=1}^r d_j X_j^2,$
        \item $X_1 > \frac{\varepsilon}{2}\mathcal{L}^2.$
    \end{itemize}
\end{proposition}
\begin{proof}
    $X_1 > X_i$ is preserved for $i=2, \ldots, r$ as 
\begin{align*}
    \frac{d}{ds} (X_1 - X_i) = (X_1 - X_i) \left( \sum_{j=1}^2 d_j X_j^2 - \frac{\varepsilon}{2} \mathcal{L}^2 - 1 \right) + (d_1-1) Y_1^2
\end{align*}
as $\mu_i=0$ for $i=2, \ldots, r$ and $Y_1 >0$.

The second statement now follows immediately from the first, using $||\cdot||_2^2 \le ||\cdot||_\infty \cdot ||\cdot||_1$ on the vector $X=(X_1, \ldots , X_1, X_2, \ldots , X_2, \ldots, X_r, \ldots, X_r)$ and that $||X||_1 = \sum d_i X_i \le 1$ with equality only along Einstein trajectories. In particular, the claim follows for regular trajectories. Along Einstein trajectories, we also have $||X||_2^2 < ||X||_{\infty}$ as equality is only possible if all nonzero $X_i$ are equal, which is impossible since $X_1 > X_2 > 0$.

For the last claim, it is immediate to compute that
\begin{align*}    
\frac{d}{ds} \left( X_1 - \frac{\varepsilon}{2} \mathcal{L}^2 \right) = & \ \left( X_1 - \frac{\varepsilon}{2} \mathcal{L}^2 \right) \left( \sum_{j=1}^r d_j X_j^2 - \frac{\varepsilon}{2} \mathcal{L}^2 - 1 \right)  \\
& \ + (d_1-1) Y_1^2 - \frac{\varepsilon}{2} \mathcal{L}^2 \left( \sum_{j=1}^r d_j X_j^2 - \frac{\varepsilon}{2} \mathcal{L}^2 \right).
\end{align*}
Note that whenever $X_1= \frac{\varepsilon}{2} \mathcal{L}^2$ we have 
\begin{align*}
    - \frac{\varepsilon}{2} \mathcal{L}^2 \left( \sum_{j=1}^r d_j X_j^2 - \frac{\varepsilon}{2} \mathcal{L}^2 \right) = X_1 \left( X_1 - \sum_{j=1}^r d_j X_j^2 \right) > 0
\end{align*}
and the claim follows.
\end{proof}

\begin{proposition}
\label{QuotientYX}
    Consider a trajectory of the Ricci soliton ODE with $X_1, Y_1 >0,$ $X_1 \geq X_i \geq 0$ for $i=2, \ldots, r$ and $\mathcal{L} \geq 0.$ Set 
\begin{align*}
    Z = \frac{1}{X_1} \left( (d_1-1)Y_1^2 + \frac{\varepsilon}{2} \mathcal{L}^2 \right).
\end{align*}
    \begin{enumerate}
        \item If $\frac{Y_1}{X_1} > \sqrt{ \frac{n-1}{d_1-1}}$ and $1-X_1-Z>0$ at $s_0 \in \R,$ then $\frac{Y_1}{X_1}$ is non-decreasing for $s \geq s_0.$
        \item If $\frac{Y_1}{X_1} < 1$ and $1-X_1-Z<0$ at $s_0 \in \R,$ then $\frac{Y_1}{X_1}$ is non-increasing for $s \geq s_0.$
    \end{enumerate}
\end{proposition}
\begin{proof}
Note that 
    \begin{align*}
        \frac{d}{ds} \frac{Y_1}{X_1} = \frac{Y_1}{X_1} \left(1 - X_1 - Z \right)
    \end{align*}
and 
\begin{align*}
    (1-X_1-Z)^{'} =  \left( \sum_{j=1}^r d_j X_j^2 - \frac{\varepsilon}{2} \mathcal{L}^2 \right) (1-X_1-Z) - \sum_{j=1}^r d_j X_j^2 + X_1 + (d_1-1)Y_1^2 - Z + Z^2.
\end{align*}
In particular, 
\begin{align*}
    (1-X_1-Z)^{'}_{|Z=1-X_1} &  = - \sum_{j=1}^r d_j X_j^2 + (d_1-1)Y_1^2 + X_1^2.
\end{align*}

In case (a) it follows that 
\begin{align*}
    (1-X_1-Z)^{'}_{|Z=1-X_1} \geq X_1^2 \left( -(n-1) + (d_1-1) \left( \frac{Y_1}{X_1} \right)^2 \right) \geq 0.
\end{align*}
Thus the term $1-X_1-Z$ remains nonnegative and thus $\frac{Y_1}{X_1}$ is non-decreasing. 

In case (b) observe that
\begin{align*}
    (1-X_1-Z)^{'}_{|Z=1-X_1} \leq (d_1-1) X_1^2 \left( \left( \frac{Y_1}{X_1} \right)^2 -1 \right) \leq 0.
\end{align*}
Therefore $1-X_1-Z$ remains nonpositive and $\frac{Y_1}{X_1}$ is non-increasing. 
\end{proof}

\begin{remark}
\label{SigmaAlongLimitTrajectories}
    \normalfont
    Note that the condition $1-X_1-Z>0$ is equivalent to $\frac{d}{ds} \frac{Y_1}{X_1}>0.$ It follows that the trajectory $\gamma$ of the subsystem in section \ref{SectionSubsystem} eventually satisfies part (a) of Proposition \ref{QuotientYX} as $\frac{Y}{X} \to \infty$ as $s \to \infty.$

    Similarly, Einstein trajectories eventually satisfy part (b) of Proposition \ref{QuotientYX} due to the asymptotics of section \ref{SectionAsymptoticsMetric}.

\end{remark}

\begin{remark}
    \normalfont
    Note that $Z \to 1$ both along regular trajectories as well as along trajectories of the subsystem of section \ref{SectionSubsystem}. In fact, it follows 
    from $(Z-1)_{|Z=1}^{'} \leq nX_1^2 - X_1 \leq 0$ that $Z-1 \leq 0$ is preserved. Moreover, an application of lemma \ref{LemmaAsymptotics} to 
    \begin{align*}
    \frac{d}{ds} \frac{Z-1}{X_1} = \frac{Z-1}{X_1} (1 -2Z )- Z - (d_1-1) \frac{Y_1^2}{X_1} + \frac{\sum_{j=1}^r d_j X_j^2}{X_1}
    \end{align*}
    shows that $\frac{Z-1}{X_1} \to -1$ as $s \to \infty.$ 
\end{remark}

By quantifying the strategy of the proof of lemma \ref{AsymptoticLemma}, we obtain continuity of $\sigma_1,$ which is the last ingredient for the proof of lemma \ref{Simga1Realized}.

\begin{proposition}
\label{Sigma1Continuous}
    $\sigma_1$ is continuous on the set of regular trajectories.  
\end{proposition}
\begin{proof}

    Pick some regular trajectory $\gamma$ and fix some $T>0$. Pick some small compact neighborhood $K$ of $\gamma(T)$. We want to establish that the constant of equation \eqref{eqAsBounds} may be chosen uniformly for all trajectories passing through $K$.
    
    Since $\frac{Y_1}{\mathcal{L}^2}$ is always decreasing by proposition \ref{PreservedInequalities}, it is bounded in terms of its value at $\gamma(T)$ for all trajectories passing through $K$ and times $t\ge T$. Plugging this into the differential equation for $\frac{X_i}{\mathcal{L}^2}$ we get a priori bounds for $\frac{X_i}{\mathcal{L}^2}$ as well, using that $\mu_i=0$ and that the conservation law implies $\frac{\varepsilon}{2}\mathcal{L}^2\leq X_1 \leq \frac{1}{d_1}, Y_1\le 1$.
    This in turn gives a priori bounds for $\frac{X_1-\frac{\varepsilon}{2}\mathcal{L}^2}{\mathcal{L}^4}$.

    With $\frac{d}{ds} = \mathcal{L}\frac{d}{dt}$ we find that 
    \begin{align*}
        \frac{d}{ds}t\mathcal{L} &= \mathcal{L}^2 + t\mathcal{L} \left( \sum_{i=1}^r d_i X_i^2 - \frac{\varepsilon}{2}\mathcal{L}^2 \right) \\
        &\le \mathcal{L}^2 + t\mathcal{L} \left( c(\gamma, T) \mathcal{L}^4 - \frac{\varepsilon}{2}\mathcal{L}^2 \right).
    \end{align*}

    We may choose $T'>T$ such that $\gamma(T')$ is close enough to the origin such that lemma \ref{OriginAttractor} applies. Hence we may assume $c\mathcal{L}^2<\frac{\varepsilon}{4}$ for all $t>T'$ and all trajectories passing through $K$ (after possibly shrinking $K$). This gives

    \begin{align*}
        \frac{d}{ds}t\mathcal{L} \le \mathcal{L}^2 \left(1 - \frac{\varepsilon}{4}t\mathcal{L} \right),
    \end{align*}
    which gives a uniform bound for $t\mathcal{L}$ depending only on its value at $\gamma(T')$.

    In particular, we do obtain that the constant of equation \eqref{eqAsBounds} may be chosen uniformly for trajectories passing through $K$. We now deduce continuity $\sigma_1$ at $\gamma$.

    For $\alpha_1 >0$ we find $T_1$ such that $\frac{Y_1}{X_1}(\gamma(t)) \leq (1 + \alpha_1 ) \sigma_1(\gamma)$ for all $t \geq T_1.$

    Note that the comparison solution $g^{\pm}(t)$ of
    \begin{align*}
        \frac{dg^{\pm}}{dt}=\pm c \frac{g^{\pm}}{t^3}, \ g^{\pm}(T_0)=g_0
    \end{align*}
    converges for each $T_0$ and that the limit converges to $g_0$ as $T_0 \to \infty$. In particular, for $\alpha_2>0$ there is $T_2>0$ such that $\lim_{t \to \infty} g^{+}(t) \leq (1+\alpha_2) g_0$ for $T_0 \geq T_2.$

    Choose $T\ge T_1, T_2$. By continuity, for each $\alpha_3>0$ there is a compact neighborhood $K$ of $\gamma(T)$ such that 
    \begin{align*}
        \frac{Y_1}{X_1}(\tilde{\gamma}(T)) \leq (1 + \alpha_3 )  \frac{Y_1}{X_1}(\gamma(T))
    \end{align*}
    for all $\tilde{\gamma}$ with $\tilde{\gamma}(T) \in K.$

    If $g^+$ denotes the comparison solution with $g^{+}(T)=\frac{Y_1}{X_1}(\Tilde{\gamma}(T)),$ then  
    \begin{align*}
        \sigma_1( \Tilde{\gamma} ) \leq \lim_{t \to \infty} g^{+}(t) \leq (1+ \alpha_1)(1+\alpha_2)(1+\alpha_3) \sigma_1(\gamma).
    \end{align*}
    The lower bound follows analogously using $g^{-}$.
\end{proof}

\begin{proof}[Proof of Lemma \ref{Simga1Realized}]

Fix $\Bar{f}_2, \ldots, \Bar{f}_r>0.$ In proposition \ref{Sigma1Continuous} we proved that $\sigma_1$ depends continuously on $C<0.$ Therefore it suffices to show that $\sigma_1 \to \infty$ as $C \to -\infty$ and $\sigma_1 \to 0$ as $C \to 0.$

Pick $\Bar{\sigma}>\sqrt{\frac{n-1}{d_1-1}}.$ According to remark \ref{SigmaAlongLimitTrajectories}, we can parametrize the limit system trajectory $\gamma$ of section \ref{SectionSubsystem} such that it satisfies $\frac{Y_1}{X_1}>\sigma_0$ and $1-X_1-Z>0$ for $s\ge 0.$

Note that for any $\Bar{f}_2, \ldots, \Bar{f}_r>0$ the Ricci soliton system converges to the limiting system as $C \to - \infty.$ By continuous dependence on initial conditions we may thus pick $\delta_0 >0$ and $C_0<0$ such that the trajectory corresponding to the initial conditions $(\Tilde{f}_2, \ldots, \Tilde{f}_r, C)$ passes through a small neighborhood of $\gamma(0)$ for all $\Tilde{f}_i \in (\Bar{f}_i- \delta_0, \Bar{f}_i+\delta_0)$ and $C<C_0$. In particular, we may choose parametrizations such that they are within this neighborhood at $s=0$. We may then assume that $\frac{Y_1}{X_1}(0)>\Bar{\sigma}$ and $(1-X_1-Z)(0)>0$ for all these curves. Proposition \ref{QuotientYX} shows that $\frac{Y_1}{X_1}$ is non-decreasing along all of these curves for $s \geq 0$ and thus the limit $\sigma_1$ must be larger than $\Bar{\sigma}.$

Similarly, to show that $\sigma_1 \to 0$ as $C \to 0,$ consider for $\Bar{f}_2, \ldots, \Bar{f}_r>0$ the trajectory with initial condition $(\Bar{f}_2, \ldots, \Bar{f}_r,0)$, which is Einstein. In particular, according to section \ref{SectionAsymptoticsMetric}, $Y_1 \to 0$ while $X_1 \to \frac{1}{n}$ as $t \to \infty$. By continuous dependence on initial conditions, given $\varepsilon_0>0$ we find $\delta_0>0$ and $C_0<0$ such that for all $\Tilde{f}_i \in (\Bar{f}_i - \delta_0, \Bar{f}_i +\delta_0)$ and $C_0<C<0$ the trajectory corresponding to the initial condition $(\Tilde{f}_2, \ldots, \Tilde{f}_r, C)$ satisfies $\frac{Y_1}{X_1}<\varepsilon_0$ as well as $1-X_1-Z<0$ at some $t=T>0$. By proposition \ref{QuotientYX}, $\frac{Y_1}{X_1}<\varepsilon_0$ is preserved and thus $\sigma_1 < \varepsilon_0.$ 
\end{proof}


\begin{thebibliography}{Flat}

\bibitem[AK22]{AngenentKnopfRSConicalSingNonuniqueness}
Sigurd~B. Angenent and Dan Knopf, \emph{Ricci solitons, conical singularities,
  and nonuniqueness}, Geom. Funct. Anal. \textbf{32} (2022), no.~3, 411--489.

\bibitem[BC23]{BamlerChenDegreeTheoryRS}
Richard~H. Bamler and Eric Chen, \emph{{Degree theory for $4$-dimensional
  asymptotically conical gradient expanding solitons}}, arXiv:2305.03154
  (2023).

\bibitem[BDGW15]{BDGWExpandingSolitons}
Maria Buzano, Andrew~S. Dancer, Michael Gallaugher, and McKenzie Wang,
  \emph{Non-{K}\"{a}hler expanding {R}icci solitons, {E}instein metrics, and
  exotic cone structures}, Pacific J. Math. \textbf{273} (2015), no.~2,
  369--394.

\bibitem[BDW15]{BDWSteadySolitons}
M.~Buzano, A.~S. Dancer, and M.~Wang, \emph{A family of steady {R}icci solitons
  and {R}icci-flat metrics}, Comm. Anal. Geom. \textbf{23} (2015), no.~3,
  611--638.

\bibitem[B{\"{o}}h99]{BohmNonCompactComhomOneEinstein}
Christoph B{\"{o}}hm, \emph{Non-compact cohomogeneity one {E}instein
  manifolds}, Bull. Soc. Math. France \textbf{127} (1999), no.~1, 135--177.

\bibitem[Buz11]{BuzanoInitialValueGRS}
Maria Buzano, \emph{Initial value problem for cohomogeneity one gradient
  {R}icci solitons}, J. Geom. Phys. \textbf{61} (2011), no.~6, 1033--1044.

\bibitem[Cao97]{CaoLimitsOfSolutionsKRF}
Huai-Dong Cao, \emph{Limits of solutions to the {K}\"{a}hler-{R}icci flow}, J.
  Differential Geom. \textbf{45} (1997), no.~2, 257--272.

\bibitem[Car81]{CarrCenterManifoldTheory}
Jack Carr, \emph{Applications of centre manifold theory}, Applied Mathematical
  Sciences, vol.~35, Springer-Verlag, New York-Berlin, 1981.

\bibitem[CD20]{ConlonDeruelleExpandingConicalKRS}
Ronan~J. Conlon and Alix Deruelle, \emph{Expanding {K}\"{a}hler-{R}icci
  solitons coming out of {K}\"{a}hler cones}, J. Differential Geom.
  \textbf{115} (2020), no.~2, 303--365.

\bibitem[CDS19]{CDSclassificationExpShrKRS}
Ronan~J. Conlon, Alix Deruelle, and Song Sun, \emph{{Classification results for
  expanding and shrinking gradient K\"ahler-Ricci solitons}}, arXiv:1904.00147
  (2019).

\bibitem[Che09]{ChenStrongUniquenessRF}
Bing-Long Chen, \emph{Strong uniqueness of the {R}icci flow}, J. Differential
  Geom. \textbf{82} (2009), no.~2, 363--382.

\bibitem[Der16]{DeruelleSmoothingPosCurvedMetricCones}
Alix Deruelle, \emph{Smoothing out positively curved metric cones by {R}icci
  expanders}, Geom. Funct. Anal. \textbf{26} (2016), no.~1, 188--249.

\bibitem[DW09a]{DWExpandingSolitons}
Andrew~S. Dancer and McKenzie~Y. Wang, \emph{Non-{K}\"ahler expanding {R}icci
  solitons}, Int. Math. Res. Not. IMRN (2009), no.~6, 1107--1133.

\bibitem[DW09b]{DWSteadySolitons}
\bysame, \emph{Some new examples of non-{K}\"ahler {R}icci solitons}, Math.
  Res. Lett. \textbf{16} (2009), no.~2, 349--363.

\bibitem[DW11]{DWCohomOneSolitons}
\bysame, \emph{On {R}icci solitons of cohomogeneity one}, Ann. Global Anal.
  Geom. \textbf{39} (2011), no.~3, 259--292.

\bibitem[FIK03]{FIKSolitons}
Mikhail Feldman, Tom Ilmanen, and Dan Knopf, \emph{Rotationally symmetric
  shrinking and expanding gradient {K}\"{a}hler-{R}icci solitons}, J.
  Differential Geom. \textbf{65} (2003), no.~2, 169--209.

\bibitem[GK04]{GKExpandingRS}
Andreas Gastel and Manfred Kronz, \emph{A family of expanding {R}icci
  solitons}, Variational problems in {R}iemannian geometry, Progr. Nonlinear
  Differential Equations Appl., vol.~59, Birkh{\"a}user, Basel, 2004,
  pp.~81--93.

\bibitem[GS18]{GianniotisSchulzeRFIsolatedConicalSingularities}
Panagiotis Gianniotis and Felix Schulze, \emph{Ricci flow from spaces with
  isolated conical singularities}, Geom. Topol. \textbf{22} (2018), no.~7,
  3925--3977.

\bibitem[Ham95]{HamiltonFormationOfSingularitiesRF}
Richard~S. Hamilton, \emph{The formation of singularities in the {R}icci flow},
  Surveys in differential geometry, {V}ol. {II} ({C}ambridge, {MA}, 1993), Int.
  Press, Cambridge, MA, 1995, pp.~7--136.

\bibitem[M{\'a}x14]{MaximoBlowUpFourDimRF}
Davi M{\'a}ximo, \emph{On the blow-up of four-dimensional {R}icci flow
  singularities}, J. Reine Angew. Math. \textbf{692} (2014), 153--171.

\bibitem[Per02]{PerelmanEntropyFormula}
Grisha Perelman, \emph{{The entropy formula for the Ricci flow and its
  geometric applications}}, arXiv:0211159 (2002).

\bibitem[Per03]{PerelmanRFwithSurgery}
\bysame, \emph{{Ricci flow with surgery on three-manifolds}}, arXiv:0303109
  (2003).

\bibitem[Sie13]{SiepmannRFofFRcones}
Michael Siepmann, \emph{{Ricci flows of Ricci flat cones}}, Doctoral thesis,
  ETH Zurich, Zürich, 2013.

\bibitem[SS13]{SchulzeSimonExpandersComingOutOfCones}
Felix Schulze and Miles Simon, \emph{Expanding solitons with non-negative
  curvature operator coming out of cones}, Math. Z. \textbf{275} (2013),
  no.~1-2, 625--639.

\bibitem[Win21]{WinkSolitonsViaEstimatesOnPotential}
Matthias Wink, \emph{Complete {R}icci solitons via estimates on the soliton
  potential}, Int. Math. Res. Not. IMRN (2021), no.~6, 4487--4521.

\end{thebibliography}

\end{document}